\documentclass[11pt]{amsart}

\makeatletter
\@namedef{subjclassname@2020}{%
  \textup{2020} Mathematics Subject Classification}
\makeatother

\usepackage{tikz}
\usetikzlibrary{positioning}
\usetikzlibrary{arrows}

\usepackage{amsmath,amsfonts}
\usepackage{mathrsfs,nccmath,amssymb,amsthm,mathtools}
\usepackage{mathabx}
\usepackage{enumitem}
\usepackage{appendix}
\usepackage[all]{xy}
\xyoption{line}
\usepackage{color}
\usepackage{chngcntr}
\counterwithin{figure}{section}
\usepackage[normalem]{ulem}
\renewcommand{\le}{\varleq}
\renewcommand{\ge}{\vargeq}

\usepackage[fontsize=11pt]{scrextend}
\usepackage{stackengine}
\usepackage{filecontents}
\usepackage{hyperref}
\usepackage[normalem]{ulem}
\renewcommand{\le}{\leq}
\renewcommand{\ge}{\geq}

\newcommand{\alignb}{\[\begin{aligned} }
\newcommand{\alignn}{ \end{aligned}\]}

\newcommand\mL{L\kern-0.08cm\char39}

\newcommand{\seb}{\{\,}
\newcommand{\sen}{\,\}}

\newcommand{\clos}[1]{\mkern 1.5mu\overline{\mkern-1.5mu#1\mkern-1.5mu}\mkern 1.5mu}

\newcommand{\Mcal}{\mathcal{M}}

\newcommand{\kuu}{\emptyset}
\newcommand{\nekuu}{\neq \kuu}
\newcommand{\iskuu}{= \kuu}

\newcommand{\pdirectional}{\raise0.05em\hbox{$+$}directional}
\newcommand{\pdirectionality}{\raise0.05em\hbox{$+$}directionality}
\newcommand{\pdirectionalitys}{\raise0.05em\hbox{$+$}directionality }
\newcommand{\pdirectionals}{\raise0.05em\hbox{$+$}directional }
\newcommand{\mdirectional}{\raise0.05em\hbox{$-$}directional}
\newcommand{\mdirectionality}{\raise0.05em\hbox{$-$}directionality}
\newcommand{\mdirectionalitys}{\raise0.05em\hbox{$-$}directionality }
\newcommand{\mdirectionals}{\raise0.05em\hbox{$-$}directional }

\newcommand{\Z}{\mathbb{Z}}

\newcommand{\R}{\mathbb{R}}

\newcommand{\bi}{\in \Z}

\newcommand{\ep}{\varepsilon}

\newcommand{\centb}{\begin{center}}
\newcommand{\centn}{\end{center}}
\newcommand{\enumb}{\begin{enumerate}}
\newcommand{\enumn}{\end{enumerate}}
\newcommand{\itemb}{\begin{itemize}}
\newcommand{\itemn}{\end{itemize}}
\numberwithin{equation}{section}
\usepackage[left=30mm,right=30mm,top=25mm,bottom=25mm]{geometry}
\usepackage{cleveref}
\setlist[enumerate,1]{label=(\alph*),ref=(\alph*)}
\setlist[enumerate,2]{label=(\arabic*),ref=(\alph{enumi}-\arabic{enumii})}
\setlist[enumerate,3]{label=(\Alph*),ref=(\roman{enumi}-\alph{enumii}-\Alph*)}
\setlist[enumerate,4]{label=(\arabic*),ref=(\roman{enumi}-\alph{enumii}-\Alph{enumiii}-\arabic*)}
\newlist{enumrm}{enumerate}{1}
\setlist[enumrm,1]{label={\rm (\roman*)},ref=(\roman*)}
\newcommand{\enumrmb}{\begin{enumrm}}
\newcommand{\enumrmn}{\end{enumrm}}
\newlist{enumsec}{enumerate}{1}
\crefalias{enumseci}{enumi} 
\setlist[enumsec,1]{label={\rm (\thesection.\alph*)},ref=(\thesection.\alph*)}
\newcommand{\enumsecb}{\begin{enumsec}[resume]}
\newcommand{\enumsecn}{\end{enumsec}}
\newlist{deepenum}{enumerate}{1}
\setlist[deepenum,1]{label=($1$:\alph*),ref=(1:\alph*)}
\newlist{Lminusoneenum}{enumerate}{1}
\setlist[Lminusoneenum,1]{label=($-1$:\alph*),ref=($-1$:\alph*)}
\newlist{Ldeepenum}{enumerate}{1}
\setlist[Ldeepenum,1]{label=($2$:\alph*),ref=($2$:\alph*)}
\numberwithin{equation}{section}
\newtheorem{thm}[equation]{Theorem}%[section]
\newtheorem{lem}[equation]{Lemma}
\newtheorem{prop}[equation]{Proposition}
\newtheorem{cor}[equation]{Corollary}

\theoremstyle{definition}
\newtheorem{defn}[equation]{Definition}

\newtheorem{example}[equation]{Example}

\theoremstyle{remark}

\newtheorem{nota}[equation]{Notation}
\newtheorem{rem}[equation]{Remark}

\crefname{sec}{\S}{\S\S}
\crefname{mainthm}{Theorem}{Theorems}
\crefname{maincor}{Corollary}{Corollaries}
\crefname{thm}{Theorem}{Theorems}
\crefname{lem}{Lemma}{Lemmas}
\crefname{prop}{Proposition}{Propositions}
\crefname{cor}{Corollary}{Corollaries}
\crefname{defn}{Definition}{Definitions}
\crefname{conj}{Conjecture}{Conjectures}
\crefname{example}{Example}{Examples}
\crefname{nota}{Notation}{Notations}
\crefname{rem}{Remark}{Remarks}
\crefname{note}{Note}{Notes}
\crefname{case}{Case}{Cases}
\crefname{figure}{Figure}{Figures}
\crefname{section}{\S}{\S\S}
\crefname{enumi}{}{}
\crefname{enumii}{}{}
\crefname{equation}{}{}

\newcommand{\abs}[1]{|#1|}

\newcommand{\imply}{$\Rightarrow$}

\begin{document}
  
\title[On some properties of basic sets]{On some properties of basic sets}

\author{Takashi Shimomura}

\address{Nagoya University of Economics, Uchikubo 61-1, Inuyama 484-8504, Japan}
\curraddr{}
\email{tkshimo@nagoya-ku.ac.jp}
\thanks{}

\subjclass[2020]{Primary 37B02, 37B05.}

\date{\today}

\dedicatory{}

\commby{}

\begin{abstract}
In the theory of zero-dimensional systems
 and their relation to $C^*$-algebras,
 Poon (1990) introduced a class of closed sets.
We call the closed sets quasi-sections.
Medynets (2006) introduced basic sets that are part of quasi-sections
 in his study of aperiodic zero-dimensional systems
 and their relation to Bratteli--Vershik
 models and $C^*$-algebras.
Downarowicz and Karpel (2019)
 introduced the notion of decisiveness in the theory of 
 Bratteli--Vershik models.
We previously clarified that particular quasi-sections can be
 the ``bases'' of the decisive Bratteli--Vershik models
 for zero-dimensional systems with dense aperiodic orbits.
We call them continuously decisive quasi-sections.
However, even the basic topological properties
 of quasi-sections and the basic sets
 have not been studied systematically.
This paper presents such a systematic study.
Some properties are defined, stated, and proved
 in the general settings of
 compact Hausdorff topological dynamics.
For example,
if a topological dynamical system has dense aperiodic orbits
 and no wandering points,
 then every basic set is continuously decisive.
If a zero-dimensional system has 
 dense aperiodic orbits, then there exists a minimal
 continuously decisive basic set such that for every
 minimal set, there exists a unique common point.
\end{abstract}
\keywords{quasi-section, basic set, zero-dimensional systems, Bratteli--Vershik model}

\maketitle

\section{Introduction}
In this paper, we represent $(X,f)$ as a \textit{topological dynamical system}
 if $X$ is a non-empty compact Hausdorff space
 and $f : X \to X$ is a homeomorphism
 (i.e., only invertible topological dynamical systems are considered).
In particular, if $X$ is a compact metrizable zero-dimensional space,
 then $(X,f)$ is called a \textit{zero-dimensional system}.
The quasi-sections that are considered
 in \cite{Poon1990AFSubAlgebrasOfCertainCrossedProducts}
 are found again
 in \cite{Shimomura_2020ArXivBratteliVershikModelsFromBasicSets}
 to be exactly the ``bases''
 of Bratteli--Vershik models of zero-dimensional systems
 (see \cite[Theorems 4.4 and 4.5]{Shimomura_2020ArXivBratteliVershikModelsFromBasicSets}).
In the study of the Bratteli--Vershik models,
 quasi-sections appear to have
 some essential roles.
However, systematic studies have not been conducted (see, for example, 
 \cite{Poon1990AFSubAlgebrasOfCertainCrossedProducts,Medynets_2006CantorAperSysBratDiag,DownarowiczKarpel_2018DynamicsInDimensionZeroASurvey,DownarowiczKarpel_2019DecisiveBratteliVershikmodels,Shimomura_2020AiMBratteliVershikModelsAndGraphCoveringModels,Shimomura_2020ArXivBratteliVershikModelsFromBasicSets}).
Although the Bratteli--Vershik models are related to
 the study of $C^*$-algebras,
 in this study, we focus on certain purely topological aspects
 of quasi-sections and basic sets that are not ``thick.''
To understand the meaning of this thickness,
 we must note the construction of the decisive Bratteli--Vershik models of
 Bratteli--Vershikizable zero-dimensional systems that do not have
 isolated periodic orbits.
Quasi-sections need to have empty interiors for such systems
(see \cite[Proposition 1.2 and Theorem 3.1]{DownarowiczKarpel_2019DecisiveBratteliVershikmodels}).
The above-mentioned work also clarified that
 such zero-dimensional systems have dense sets of aperiodic orbits.
We introduced quasi-sections in \cite{Shimomura_2020ArXivBratteliVershikModelsFromBasicSets} and showed that
 in such zero-dimensional systems, minimal (with respect to inclusion)
 quasi-sections have empty interiors (see
 \cite[Proposition 2.22]{Shimomura_2020ArXivBratteliVershikModelsFromBasicSets}).
In this manner, in \cite{Shimomura_2020ArXivBratteliVershikModelsFromBasicSets},
 a triple $(X,f,B)$ of a zero-dimensional system $(X,f)$ and a quasi-section
 $B$ is said to be ``continuously decisive'' when the interior of $B$ is the
 empty set. 
 If $(X,f)$ is fixed, then the quasi-section $B$ with an empty interior
 is also said to be
 ``continuously decisive'' (see \cref{defn:condec}).
In the same paper, we showed that
 there exists a one-to-one correspondence between the equivalence classes of
 continuously decisive ordered Bratteli diagrams
 and the topological conjugacy classes of continuously decisive triples
 of zero-dimensional systems with quasi-sections
 (see \cite[Theorem 5.7]{Shimomura_2020ArXivBratteliVershikModelsFromBasicSets}).
Therefore, the continuous decisiveness condition for quasi-sections is valuable
 to study.

On the other hand, Medynets \cite{Medynets_2006CantorAperSysBratDiag} 
 defined the basic sets for aperiodic zero-dimensional systems and showed their 
 existence for such systems.
This class of basic sets is a subclass of the class of quasi-sections.
In \cite{Shimomura_2020AiMBratteliVershikModelsAndGraphCoveringModels,Shimomura_2020ArXivBratteliVershikModelsFromBasicSets},
 we have defined an analogy of basic sets for all
 (invertible) zero-dimensional systems and have shown their existence 
 for all zero-dimensional systems (see \cref{defn:basicset,thm:extremal}).
Furthermore, we have shown a one-to-one correspondence between
 the equivalence classes of ordered Bratteli diagrams with a closing property
 and the topological conjugacy
classes of continuously decisive triples of zero-dimensional systems
 with basic sets (see \cite[Theorem 5.9]{Shimomura_2020ArXivBratteliVershikModelsFromBasicSets}).
In addition,
 we have shown the existence of a special kind of basic sets (see \cref{thm:extremal}).
In this paper,
 these new basic sets are called ``extremal'' (cf. \cref{defn:extremal}) basic sets.
In these ways, some refinements of quasi-sections have already been introduced.
Henceforth, the investigation of their basic properties 
 and the relations between them appears to be a worthwhile pursuit.
\section{Preliminaries}
In this study, quasi-sections are investigated for
 the study of zero-dimensional systems in the form of
 Bratteli--Vershik models.
However, many definitions, propositions, and proofs
 need only topological dynamical settings.
Thus, throughout this section, $(X,f)$ is a topological dynamical system.
 If we restrict $(X,f)$ to a zero-dimensional system,
 then we explicitly state that $(X,f)$ is a zero-dimensional system.
\begin{nota}
For a subset $A$ of a topological space $X$, the set of interior points
 of $A$ is denoted by $\textrm{int}A$.
\end{nota}
\begin{nota}
For each $x \in X$, the orbit of $x$ is denoted by $O(x)$,
 i.e., $O(x) := \seb f^n(x) \mid n \bi \sen$.
We also denote it as $O^+(x) := \seb f^n(x) \mid n \ge 0 \sen$ and
 $O^-(x) := \seb f^n(x) \mid n \le 0 \sen$.
In addition, for a subset $A \subseteq X$, we denote the orbit 
 as $O(A) := \bigcup_{n \bi}f^n(A)$,
 $O^+(A) := \bigcup_{n \ge 0}f^n(A)$, and
 $O^-(A) := \bigcup_{n \le 0}f^n(A)$.
\end{nota}
Evidently, it follows that $O(A) = X$ if and only if for all $x \in X$,
 $O(x) \cap A \nekuu$.
\begin{nota}\label{nota:minimalset}
A non-empty closed set $M \subseteq X$ is called a \textit{minimal set} 
 if $f(M) = M$ and every orbit in $M$ is dense in $M$.
Using this, we denote $\Mcal_f := \seb M \mid M \text{ is a minimal set.} \sen$.
Consequently, we also denote $M_f := \clos{\bigcup_{M \in \Mcal_f} M}$.
\end{nota}

\begin{rem}
If the periodic orbits are dense in $(X,f)$, then $M_f = X$. 
\end{rem}

Let us recall the non-wandering sets for topological dynamical systems.
\begin{defn}
Let $(X,f)$ be a topological dynamical system.
A point $x \in X$ is said to be \textit{wandering} if there exists
 an open set $U \ni x$ such that $f^n(U) \cap U \iskuu$ for all $n \ne 0$.
Thus, we define  $\Omega_f := \seb x \mid x \text{ is not wandering. } \sen$.
The set $\Omega_f$ is said to be the \textit{non-wandering set} of $(X,f)$.
\end{defn}

\begin{lem}\label{lem:MfOf}
It follows that $M_f \subseteq \Omega_f$.
\end{lem}
\begin{proof}
It is evident that $M \subseteq \Omega_f$ for each $M \in \Mcal_f$.
Then, it follows that $M_f \subseteq \Omega_f$ because
 $\Omega_f$ is closed.
\end{proof}

\begin{nota}
For each $x \in X$, the $\omega$-limit set is denoted as $\omega(x)$,
 and the $\alpha$-limit set is denoted as $\alpha(x)$.
\end{nota}
\begin{nota}\label{nota:ape}
Let $(X,f)$ be a topological dynamical system.
A point $x \in X$ is said to be \textit{aperiodic}
 if $\abs{O(x)} = \infty$.
The set of aperiodic points is denoted as $A_f$, i.e.,
 $A_f = \seb x \mid f^n(x) \ne x \textrm{ for all } n \ne 0 \sen$.
\end{nota}

\begin{rem}
We note that $f(M_f) = M_f$, and $(M_f, f|_{M_f})$ is a sub-dynamical system.
\end{rem}

Downarowicz and Karpel \cite{DownarowiczKarpel_2019DecisiveBratteliVershikmodels}
 considered zero-dimensional systems
 in which aperiodic points are dense.
This property plays an important role in the present study.
\begin{defn}\label{defn:denape}
A topological dynamical system $(X,f)$ is \textit{densely aperiodic}
 if aperiodic points are dense in $X$, i.e., $\clos{A_f} = X$.
\end{defn}

In the theory of zero-dimensional systems,
 following Medynets \cite{Medynets_2006CantorAperSysBratDiag},
 we apply the following definition. 

\begin{defn}\label{defn:complete-section}
Let $(X,f)$ be a zero-dimensional system.
A clopen set $U \subseteq X$ is called a \textit{complete section}
 if for each $x \in X$, it follows that $U \cap O(x) \nekuu$.
\end{defn}

For the theory of zero-dimensional systems, we need only the following.
\begin{prop}\label{prop:open-complete-section}
Let $(X,f)$ be a zero-dimensional system and $U$ be a clopen set.
Then, the following statements are equivalent:
\enumb
\item\label{compsec:compsec} $U$ is a complete section,
\item\label{compsec:+}
 $O^+(U) = X$,
\item\label{compsec:-}
 $O^-(U) = X$,
\item\label{compsec:omega}
 $\omega(x) \cap U \nekuu$ for every $x \in X$,
\item\label{compsec:alpha}
 $\alpha(x) \cap U \nekuu$ for every $x \in X$, and
\item\label{compsec:minimal}
 $M \cap U \nekuu$ for every $M \in \Mcal_f$.
\enumn
\end{prop}

However, a proof is given in general topological dynamical settings
 through the following lemma.

\begin{lem}\label{lem:orbitopen+-}
Let $U \subseteq X$ be an open set.
Then, the following are equivalent:
\enumb
\item\label{item:orbitopen} $O(U) = X$,
\item\label{item:orbitopen+} $O^+(U) = X$,
\item\label{item:orbitopen-} $O^-(U) = X$,
\item\label{item:openalpha}
 $\alpha(x) \cap U \nekuu$ for every $x \in X$,
\item\label{item:openomega}
 $\omega(x) \cap U \nekuu$ for every $x \in X$, and
\item\label{item:openminimal}
 $M \cap U \nekuu$ for every $M \in \Mcal_f$.

\enumn
\end{lem}
\begin{proof}
To show \cref{item:orbitopen} $\Rightarrow$ \cref{item:openminimal},
 let $U$ be an open set with $O(U) = X$ and $M \in \Mcal_f$.
Then, every orbit in $M$ intersects with $U$, as desired.
To show \cref{item:openminimal} $\Rightarrow$ \cref{item:openalpha},
 let $x \in X$.
There exists an $M \in \Mcal_f$ such that $M \subseteq \alpha(x)$.
This concludes the proof in this case.
In the same manner,
 one can show \cref{item:openminimal} $\Rightarrow$ \cref{item:openomega}.
To show \cref{item:openalpha} $\Rightarrow$ \cref{item:orbitopen+}, let
 $U$ be an open set with the condition \cref{item:openalpha}.
Suppose that there exists an $ x \in X \setminus O^+(U)$.
Then, it follows that $O^-(x) \cap U = \kuu$.
Consequently, it follows that $\alpha(x) \cap U = \kuu$ because
 $U$ is an open set.
This contradicts \cref{item:openalpha}.
In the same manner, we can show that
 \cref{item:openomega} $\Rightarrow$ \cref{item:orbitopen-}.
Finally, it is a direct consequence that
 \cref{item:orbitopen+} $\Rightarrow$ \cref{item:orbitopen} and
 \cref{item:orbitopen-} $\Rightarrow$ \cref{item:orbitopen}.
This concludes the proof.
\end{proof}

By the compactness of $X$, 
\cref{item:orbitopen+} is equivalent to
 $\bigcup_{i = 0}^nf^i(U) = X$ for some $n \ge 0$, and
\cref{item:orbitopen-} is equivalent to
 $\bigcup_{i = 0}^nf^{-i}(U) = X$ for some $n \ge 0$.

\vspace{3mm}

\noindent \textit{Proof of \cref{prop:open-complete-section}.}

We can directly obtain the proof from \cref{lem:orbitopen+-}.
\qed

\section{Quasi-section}

In this section, we present a basic study of the quasi-sections.
In the study of
 the $C^*$-algebras of zero-dimensional systems, Poon \cite{Poon1990AFSubAlgebrasOfCertainCrossedProducts} considered
 closed sets such that every clopen neighborhood is a complete section.
We introduce
 the following definition in the case of general topological dynamical systems.
\begin{defn}
A closed set $A \subseteq X$ is a \textit{quasi-section}
 if every open set $U \supseteq A$ satisfies $O(U) = X$.
\end{defn}

The following are the most basic properties of the quasi-sections in this study:
\begin{prop}\label{prop:qs}
Let $(X,f)$ be a topological dynamical system and 
 $A \subseteq X$ be a closed set.
Then, the following statements are equivalent:
\enumb
\item\label{qs:qs} $A$ is a quasi-section,
\item\label{qs:orbit} $\clos{O(x)} \cap A \nekuu$ for every $x \in X$,
\item\label{qs:omega} $\omega(x) \cap A \nekuu$ for every $x \in X$,
\item\label{qs:alpha} $\alpha(x) \cap A \nekuu$ for every $x \in X$, and
\item\label{qs:minimal} $A \cap M \nekuu$ for every $M \in \Mcal_f$.
\enumn
\end{prop}
\begin{proof}
To show \cref{qs:qs} \imply \cref{qs:minimal}, let $M \in \Mcal_f$
and $U \supseteq A$ be an open set.
Then, by the definition of the quasi-section, it follows that $O(U) = X$.
By \cref{item:openminimal} of \cref{lem:orbitopen+-},
 it follows that $U \cap M \nekuu$.
It then follows that $A \cap M \nekuu$
 because both $A$ and $M$ are closed sets.
Next, the implications \cref{qs:minimal} \imply \cref{qs:alpha} and
 \cref{qs:minimal} \imply \cref{qs:omega} are evident because
 all $\alpha$-limit sets and $\omega$-limit sets contain minimal sets.
The implications \cref{qs:alpha} \imply \cref{qs:orbit} and
 \cref{qs:omega} \imply \cref{qs:orbit} are evident because
 all orbit closures contain an $\alpha$-limit set and 
 $\omega$-limit set.
Finally, by the definition of the quasi-section,
 the implication \cref{qs:orbit} \imply \cref{qs:qs} is evident.
This completes the proof.
\end{proof}

The minimality of the quasi-sections plays important roles in the present study.

\begin{nota}We say that a quasi-section $A$ is \textit{minimal}
 if $A$ is minimal with respect to the inclusion of sets.
\end{nota}
\begin{rem}
If $(X,f)$ is itself minimal, i.e., $\seb X \sen = \Mcal_f$, then
 a quasi-section is minimal if and only if it consists of a single point.
\end{rem}

\begin{rem}
The existence of minimal quasi-sections had been observed
 in \cite[\S 4]{Poon1990AFSubAlgebrasOfCertainCrossedProducts} in the 
 zero-dimensional case.
However, this fact is still valid in general topological dynamical settings.
Concretely, we obtain the following lemma:
\end{rem}

\begin{lem}
Every quasi-section contains a minimal quasi-section.
\end{lem}
\begin{proof}
The proof follows directly from \cref{qs:minimal} of \cref{prop:qs}
 in combination with
 the assumption that $X$ is a compact Hausdorff space. 
\end{proof}

\begin{prop}\label{prop:minimalqsMf}
Every minimal quasi-section is contained in $M_f$.
\end{prop}
\begin{proof}
Let $A$ be a minimal quasi-section.
Take $x_M \in A \cap M$ for each $M \in \Mcal_f$.
Let $A' := \clos{\seb x_M \mid M \in \Mcal_f \sen}$.
Then, we evidently obtain $A' \subseteq A \cap M_f$, and $A'$ is a quasi-section
 by \cref{qs:minimal} of \cref{prop:qs}.
By the minimality of $A$, we obtain $A = A' \subseteq M_f$, as desired.
\end{proof}

A consequence of \cref{prop:minimalqsMf} is the following.
\begin{prop}\label{prop:orbitofA}
If $A$ is a minimal quasi-section, then
 $\clos{O(A)} = \clos{O^+(A)} = \clos{O^-(A)} = M_f$.
\end{prop}
\begin{proof}
Let $A$ be a minimal quasi-section.
Then, by \cref{prop:minimalqsMf}, we obtain $\clos{O(A)} \subseteq M_f$.
Conversely, for every $M \in \Mcal_f$, it follows that $A \cap M \nekuu$.
Thus, we obtain $\clos{O^+(A)} \supseteq M$ and $\clos{O^-(A)} \supseteq M$.
Consequently, we obtain $\clos{O^+(A)} \supseteq M_f$ and 
$\clos{O^-(A)} \supseteq M_f$.
The rest of the proof is self-evident.
\end{proof}

\begin{prop}
Every minimal quasi-section is contained in $\Omega_f$.
\end{prop}
\begin{proof}
We obtain the proof from \cref{prop:minimalqsMf,lem:MfOf}.
\end{proof}

Finally, in this section, we introduce one notion on the quasi-sections
 that is brought about in the theory of Bratteli--Vershik models.
Downarowicz and Karpel \cite{DownarowiczKarpel_2019DecisiveBratteliVershikmodels},
 introduced the decisiveness
 of the ordered Bratteli diagrams and introduced the Bratteli--Vershikizability
 condition
 for zero-dimensional systems.
In one of their main results,
 they clarified that densely aperiodic zero-dimensional systems are
 essential
 parts of Bratteli--Vershikizable zero-dimensional systems
 (\cite[Theorem 3.1]{DownarowiczKarpel_2019DecisiveBratteliVershikmodels}).
They also clarified that for the densely aperiodic systems,
 the emptiness of the interiors of the ``base'' of the Bratteli--Vershik model
 is important
(see \cite[Proposition 1.2]{DownarowiczKarpel_2019DecisiveBratteliVershikmodels}).
In \cite{Shimomura_2020ArXivBratteliVershikModelsFromBasicSets}, we introduced
 the following notion in the case of zero-dimensional systems.
In the present paper, we introduce it in the general topological dynamical settings
 because many of the related arguments can be logically made in these settings.
\begin{defn}\label{defn:condec}
A triple $(X,f,B)$ of a topological dynamical system $(X,f)$
 and a quasi-section is
 \textit{continuously decisive} if $\textrm{int}B = \kuu$.
If a topological dynamical system $(X,f)$ is fixed, then
 we say that a quasi-section $B$ is also \textit{continuously decisive} if
 $\textrm{int}B = \kuu$.
\end{defn}
In \cite[Theorem 5.7]{Shimomura_2020ArXivBratteliVershikModelsFromBasicSets},
 we have shown that
 there exists a one-to-one correspondence
 between the equivalence classes of
continuously decisive ordered Bratteli diagrams
 and the topological conjugacy classes of continuously decisive triples
 of zero-dimensional systems with quasi-sections.
%Although we have not gotten
% a related result in the general topological dynamical
% settings, some properties of quasi-sections can be shown in the general
% settings as in the following proposition and \cref{sec:basicset}.
\begin{prop}\label{prop:minimalqscontdic}
Let $(X,f)$ be a densely aperiodic topological dynamical system.
Then, every minimal quasi-section is continuously decisive.
\end{prop}
\begin{proof}
Let $B$ be a minimal quasi-section.
Suppose that $U := \textup{int}B \nekuu$.
First, suppose that there exists an $x \in U$ and an $n \ne 0$ such that
 $f^n(x) = x$.
Then, there exists an aperiodic point $y \in U$ that is close to $x$ such that
 $f^n(y) \in U$ because $(X,f)$ is densely aperiodic by the assumption.
Henceforth, there exists an open set $V$ $(y \in V \subseteq U)$ such that
  $V \cap f^n(V) \iskuu$ and $f^n(V) \subseteq U$.
In this case, if we define $B' := B \setminus V$, then 
 $B'$ is a quasi-section because every orbit that passes through $V$ also
 passes through $f^n(V) \subseteq B'$.
This contradicts the minimality of $B$.
Next, suppose that there does not exist any periodic point in $U$.
We shall show that $f^n(U) \cap U \iskuu$ for all $n \ne 0$.
Suppose that there exists an $n \ne 0$ with $f^n(U) \cap U \nekuu$.
Then, there exists an $x \in U$ with $f^n(x) \in U$.
It follows that $x \ne f^n(x)$ because no periodic point exists in $U$.
Therefore, there exists an open set $V$ ($x \in V \subset U$) with
 $V \cap f^n(V) \iskuu$ and $f^n(V) \subseteq U$.
Again, we have a contradiction as in the first case.
Therefore, we obtain $f^n(U) \cap U \iskuu$ for all $n \ne 0$.
This shows that $U \cap M \iskuu$ for all $M \in \Mcal_f$.
Thus, $B' := B \setminus U$ is a quasi-section.
This contradicts the minimality of $B$.
This completes the proof.
\end{proof}

\section{Basic set}\label{sec:basicset}
The basic sets are the quasi-sections
 that are strictly restricted as in \cref{defn:basicset}.
As in the case of the quasi-sections, many studies can be conducted 
 in the general topological dynamics.
Medynets \cite{Medynets_2006CantorAperSysBratDiag} defined basic sets 
 for aperiodic zero-dimensional systems.
Following this, we present
the following definition for general topological dynamics.

\begin{defn}\label{defn:basicset}
Let $(X,f)$ be a topological dynamical system.
A closed set $A \subseteq X$ is called a \textit{basic set}
 if $A$ is a quasi-section and for every $x \in X$, $\abs{O(x) \cap A} \le 1$.
\end{defn}

\begin{rem}
It is natural to state that
 a basic set is continuously decisive if $\textup{int}B \iskuu$
 because a basic set is a quasi-section.
\end{rem}

\begin{prop}\label{prop:bs}
Let $A \subseteq X$ be a closed set.
Then, the following statements are equivalent:
\enumb
\item\label{bs:bs} $A$ is a basic set,
\item\label{bs:orbit} for every $x \in X$, $\abs{O(x) \cap A} \le 1$
 and $\clos{O(x)} \cap A \nekuu$,
\item\label{bs:omega} for every $x \in X$, $\abs{O(x) \cap A} \le 1$
 and $\omega(x) \cap A \nekuu$,
\item\label{bs:alpha}  for every $x \in X$, $\abs{O(x) \cap A} \le 1$
 and $\alpha(x) \cap A \nekuu$, and
\item\label{bs:minimal} for every $x \in X$, $\abs{O(x) \cap A} \le 1$,
 and for every $M \in \Mcal_f$, $M \cap A \nekuu$.
\enumn
\end{prop}
\begin{proof}
By the definition of basic sets and by \cref{prop:qs},
 the proof is self-evident.
\end{proof}

Although the existence of quasi-sections is evident
 for every topological dynamical system,
 basic sets may not exist in the general topological dynamics.

\begin{example}\label{example:rationalrotation}
A rational rotation of the circle $S^1 := \R/\Z$ has no basic set
 if the rotation number is not zero.
To confirm this, suppose that there exists a basic set $B \subseteq S^1$.
Every orbit of the rational rotation is a periodic orbit and,
 for each $a \in S^1$, $\abs{O(a) \cap B} = 1$.
Let $x \in B$.
Then, it is easy to see that for any sufficiently small $\epsilon > 0$,
 every $y \in (x - \epsilon, x+\epsilon)$ also satisfies $y \in B$;
 otherwise, $O(x) \setminus \seb x \sen$ has an accumulation point of
 $B$.
Therefore, $B$ becomes an open set.
Thus, $B$ is a closed and open set in $S^1$.
It follows that $B = S^1$, which is a contradiction.

On the contrary, an irrational rotation of $S^1$ is minimal.
Therefore, every set $\seb x \sen$ ($x \in S^1$) is a basic set.
\end{example}

\begin{example}\label{example:parametrizedrotations}
A densely aperiodic topological dynamical system may not have
 any basic set.
To confirm this, for each $a \in [0,1]$, let $\rho_a : S^1 \to S^1$ be a rotation
 with rotation number $a$.
Let $f : [0,1]\times S^1 \to [0,1] \times S^1$ be a map
 defined as $f(a,x) = (a,\rho_a(x))$ for each $(a,x) \in [0,1] \times S^1$.
Then, evidently, $f$ is densely aperiodic.
However, by the observation in \cref{example:rationalrotation}, $f$ cannot have
 any basic set.

\end{example}

For the zero-dimensional systems that are our main concern,
 the existence of basic sets has been shown
 (see \cite[Theorem 6.5]{Shimomura_2020ArXivBratteliVershikModelsFromBasicSets}).
We provide a short proof in \cref{cor:zds:existbs}.
\begin{rem}
Suppose that $(X,f)$ has a basic set and
 $x \in X$ is an arbitrary point.
Then, a basic set that contains $x$ always exists.
First, suppose that $B$ is a basic set such that $B \cap O(x) \nekuu$.
 Then, there exists an $n \bi$ such that
 $f^n(B) \ni x$.
It is easy to deduce that $f^n(B)$ is also a basic set.
Next, suppose that $B \cap O(x) \iskuu$.
Then, $B \cup \seb x \sen$ is a basic set that contains $x$.
\end{rem}

\begin{example}
Let $(X,f)$ be a minimal topological dynamical system.
Then, it is evident that every closed set $A$ is a quasi-section.
If $x \in X$, then the one-point set $\seb x \sen$ is a basic set.
If one takes a finite set $x_1, x_2, \dotsc,x_n$ from mutually distinct orbits,
then the finite set $\seb x_1, x_2, \dotsc, x_n \sen$ is a basic set.
\end{example}

%\begin{thm}
%uncountable basic set exists for every CMS
%
%\end{thm}

\begin{lem}
Every basic set contains a minimal basic set.
\end{lem}
\begin{proof}
From the definition of basic sets and the fact that every quasi-section
 contains a minimal quasi-section,
 the proof is self-evident.
\end{proof}
\begin{lem}\label{lem:pointsfromms}
Let $B$ be a basic set.
For each $M \in \Mcal_f$, take an $x_M \in B \cap M$.
It follows that $\overline{\seb x_M : M \in \Mcal \sen} \subseteq B$
 is also a basic set.
In particular, if $B$ is a minimal basic set,
 then it follows that $B = \overline{\seb x_M : M \in \Mcal_f \sen}$.
\end{lem}
\begin{proof}
Apply \cref{prop:bs}.

\end{proof}
\begin{example}
For each $M \in \Mcal_f$, take and fix an $x_M \in M$.
Then, even if the set
 $B := \clos{\seb x_M \mid M \in \Mcal_f \sen}$ is a basic
 set, $B$ may not be a minimal basic set.
To confirm this,
 let $C$ be the Cantor set and $(X, f)$ be a minimal
 topological dynamical system.
Take and fix two points $x_1,x_2 \in X$ from distinct orbits.
Let us consider a topological dynamical system
 $f\times id$ as $(f \times id) (x,y) = (f(x),y)$.
Then, $\Mcal_{(f \times id)} = \seb X \times \seb y \sen \mid y \in C \sen$.
Take and fix a $c \in C$.
Let $x_c = (x_2,c)$, and for each $y \in C$ with $y \ne c$,
 let $x_y := (x_1,y)$.
Then, $B := \clos{ \seb x_y \mid y \in C \sen} = \seb (x_1,y) \mid y \in C \sen \cup \seb (x_2,c) \sen$ is a basic set.
However, $\seb (x_1,y) \mid y \in C \sen \subsetneq B$
 is the only minimal basic set that is contained in $B$.
\end{example}

\begin{example}\label{example:minimalbs2pt}
Even if $B$ is a minimal basic set, \cref{lem:pointsfromms}
 does not imply that $\abs{B \cap M} = 1$ for all $M \in \Mcal_f$. 
To confirm this, let $(X, f)$ be a minimal set and $C$
 be the Cantor set in the interval $[0, 1]$.
Let us
 consider a zero-dimensional system
 $f\times id$ as $(f \times id) (x,y) = (f(x),y)$.
Take $x,y \in X$ from distinct orbits and take $a \in C$ such that
 neither $[0,a) \cap C$ nor $(a,1] \cap C$ is closed in $C$.
Then, the set
 $B :=
 (\seb x \sen \times (C \cap [0,a])) \cup (\seb y \sen \times (C \cap [a,1]))$
 is a minimal basic set, and 
 the set $(X \times \seb a \sen) \cap B$
 contains two points, $(x,a)$ and $(y,a)$.
\end{example}

\begin{rem}
In the above example, if $x$ and $y$ are in the same orbit, then
 the set $B$ defined above is not a basic set;
 however, it is a minimal quasi-section.
Therefore, a minimal quasi-section is not necessarily a basic set.
\end{rem}

\begin{lem}\label{lem:intBwander}
Let $(X,f)$ be a densely aperiodic topological dynamical system.
Suppose that $B$ is a basic set.
Then, for every $n \ne 0$, $f^n(\textup{int}B) \cap \textup{int}B \iskuu$.
In particular, every point $x \in \textup{int}B$ is wandering.
\end{lem}
\begin{proof}
If $\textup{int}B \iskuu$, then the statement is self-evident.
Suppose that $U := \textrm{int} B \nekuu$.
Suppose, on the contrary, that there exists an $n \ne 0$ such that
 $f^n(U) \cap U \nekuu$.
By the assumption that $\clos{A_f} = X$, it follows that
 there exists an $x \in f^n(U) \cap U \cap A_f$.
Then, it follows that $\abs{O(x) \cap B} \ge 2$, which is a contradiction.
\end{proof}

\begin{lem}\label{lem:intBnominimal}
Let $(X,f)$ be a densely aperiodic topological dynamical system.
Suppose that $B$ is a basic set.
Then, $\textup{int}B \cap M \iskuu$ for all $M \in \Mcal_f$.
\end{lem}
\begin{proof}
Suppose that $\textup{int}B \cap M \nekuu$ for some $M \in \Mcal_f$.
Take an $x \in \textup{int}B \cap M$.
Then, it follows that $x \in \omega(x)$.
This contradicts \cref{lem:intBwander}.
\end{proof}

Consequently, we obtain the following.

\begin{prop}
If $(X,f)$ is a densely aperiodic topological dynamical system with
 $M_f = X$,
then every basic set is continuously decisive.
\end{prop}
\begin{proof}
Let $(X,f)$ be a densely aperiodic topological dynamical system and
 $B$ be a basic set.
From \cref{lem:intBnominimal}, we have $\textup{int}B \cap M \iskuu$
 for all $M \in \Mcal_f$.
This implies that $\textup{int}B \iskuu$ because $M_f = X$.
\end{proof}
\begin{rem}
If a densely aperiodic topological dynamical system $(X,f)$ has a dense set of periodic orbits, then every basic set
 is continuously decisive.
\end{rem}

\begin{thm}
Let $(X,f)$ be a topological dynamical system.
Suppose that $(X,f)$ is densely aperiodic and $B$ is a basic set.
Then, $B' := B \setminus \textup{int}B$ is a continuously decisive basic set.
\end{thm}
\begin{proof}
It is clear that $B'$ is a closed set
 and $\abs{B' \cap O(x)} \le 1$ for all $x \in X$.
It is also clear that $\textup{int}B' \iskuu$.
We need to show that $B' \cap M \nekuu$ for every $M \in \Mcal_f$.
Fix an $M \in \Mcal_f$.
It then follows that $B \cap M \nekuu$; however, it also follows that
 $M \cap \textup{int}B \iskuu$ owing to \cref{lem:intBwander}.
Thus, we obtain $B' \cap M \nekuu$, as desired.
\end{proof}

\begin{thm}
Let $(X,f)$ be a topological dynamical system.
Suppose that $(X,f)$ is densely aperiodic
 and $\Omega_f = X$.
Then, every basic set is continuously decisive.
\end{thm}
\begin{proof}
Let $B$ be a basic set of $(X,f)$.
By \cref{lem:intBwander}, it follows that every point in $\textrm{int}B$ is 
 wandering.
Thus, by the assumption that $\Omega_f = X$, it follows that 
 $\textrm{int} B \iskuu$.
\end{proof}

\section{Extremal basic set}
In this section, we consider only zero-dimensional systems, which
 are our main concern.
We also show that there exist basic sets in every zero-dimensional system.
To do this, we embed $X$ into the real line $\R$, i.e.,
 $X \subset \R$. In particular,
 $X$ is linearly ordered, and the order topology coincides 
 with the original topology in $X$.

\begin{nota}
We use the notations 
$\inf_f(x) := \inf \seb y \mid y \in O(x) \sen$ and
 $\inf_f := \seb \inf_f(x) \mid x \in X \sen$.
\end{nota}

We proved a few basic properties in \cite[Lemma 6.3]{Shimomura_2020ArXivBratteliVershikModelsFromBasicSets}.
For every $x \in X$, it follows that
 $\inf_f(x) \le x$ and $\inf_f(x) \in \overline{O(x)}$.
Furthermore, we obtain the following.
\begin{lem}\label{lem:basicinf}
Let $x \in \inf_f$.
 Then, it follows that $\inf_f(x) = x$.
\end{lem}
\begin{proof}
Let $x \in \inf_f$.
Then, there exists a $y \in X$ with $x = \inf_f(y)$.
Evidently, it follows that $\inf_f(\inf_f(y)) \le \inf_f(y)$.
We need to show $\inf_f(\inf_f(y)) = \inf_f(y)$.
Suppose, on the contrary, that $\inf_f(\inf_f(y)) < \inf_f(y)$.
Then, there exists an $n \bi$ such that
 $f^n(\inf_f(y)) < \inf_f(y)$.
If one chooses an $m \bi$ such that
 $f^m(y)$ is sufficiently close to $\inf_f(y)$,
 then one obtains $f^{n+m}(y) = f^n(f^m(y)) < \inf_f(y)$.
This contradicts the definition of $\inf_f(y)$.
\end{proof}

\begin{lem}\label{lem:inforbitinf}
Let $x_n \to x$ ($n \to \infty$) be a sequence such that
 there exists a sequence $y_n$ ($n= 1,2,\dots$) with
 $\inf_f(y_n) = x_n$ ($n = 1,2,\dots$).
Then, it follows that $\inf_f(x) = x$.
\end{lem}
\begin{proof}
Let $x_n$ ($n = 1,2,\dots$) as above.
Suppose that $\inf_f(x) < x$.
Then, there exists an $i \bi$ such that $f^i(x) < x$.
Take an $\ep > 0$ such that $f^i(x) + \ep < x - \ep$.
It follows that $f^i(x_n) < f^i(x) + \ep$ for every sufficiently large $n$.
On the other hand, $x - \ep < x_n$ for every sufficiently large $n$.
Therefore, we obtain $f^i(x_n) < x_n$ for a sufficiently large $n$.
This contradicts \cref{lem:basicinf}.
\end{proof}

In \cite{Shimomura_2020ArXivBratteliVershikModelsFromBasicSets}, we obtained 
 the following.

\begin{thm}\label{thm:extremal}
Suppose that $(X,f)$ is a zero-dimensional system.
Then, the set
 $\inf_f$
 is a basic set.
\end{thm}

\begin{proof}
By \cref{lem:inforbitinf}, we know that $\inf_f$ is closed.
Next, we show that $\inf_f$ is a quasi-section.
Let $M \in \Mcal_f$.
We need to show that $\inf_f \cap M \nekuu$.
Let $x \in M$.
Then, it follows that $\inf_f(x) \in \inf_f \cap M$, as desired.
Finally, we need to show that $\abs{O(x) \cap \inf_f} \le 1$ for every
 $x \in X$.
Take $x,y \in \inf_f$ such that $O(x) \ni y$.
Then, by \cref{lem:basicinf}, 
 we obtain $\inf_f(x) = x$ and $\inf_f(y) = y$.
We need to show that $x = y$.
However, this is evident from
 $x = \inf_f(x) = \inf(O(x)) = \inf(O(y)) = \inf_f(y) = y$.
\end{proof}

Now, we have the following.
\begin{cor}\label{cor:zds:existbs}
Every zero-dimensional system has a basic set.
\end{cor}
\begin{proof}
We omit the proof because it is evident from \cref{thm:extremal}.
\end{proof}

We introduce the following definition because the basic sets that
 are obtained in \cref{thm:extremal} have some particular properties.
\begin{defn}\label{defn:extremal}
Let $(X,f)$ be an invertible zero-dimensional system.
We say that a basic set $\inf_f$ that is obtained in the manner of \cref{thm:extremal} is \textit{extremal}.
\end{defn}

\begin{thm}\label{thm:extremalunique}
If $B$ is an extremal basic set and $M \in \Mcal_f$, then
 $\abs{B \cap M} = 1$.
\end{thm}
\begin{proof}
Let $B$ be an extremal basic set and $M \in \Mcal_f$.
In particular, $X$ is linearly ordered, and the order topology
 coincides with $X$ and $B = \inf_f$.
Let $x,y \in B \cap M$.
Then, it follows that
 $x = \inf_f(x) = \inf(\clos{O(x)})
 = \inf(M) = \inf(\clos{O(y)}) = \inf_f(y) = y$, as desired.
\end{proof}

\begin{rem}
\cref{thm:extremalunique} implies that there exists a minimal basic set
 that is not extremal (see \cref{example:minimalbs2pt}).
\end{rem}

\begin{cor}\label{cor:extremaluniqueminimal}
Let $B$ be an extremal basic set.
For each $M \in \Mcal$, let $x_M$ be a unique point in $B \cap M$.
Let $B' := \overline{\seb x_M : M \in \Mcal \sen}$.
Then, $B'$ is the unique minimal basic set in $B$.
In particular, $B$ has the unique minimal basic set.
\end{cor}
\begin{proof}
Let $A$ be a minimal basic set that is contained in $B$.
By \cref{bs:minimal} of \cref{prop:bs}, it follows that 
$A \cap M \nekuu$ for each $M \in \Mcal_f$.
Thus, it follows that $A \cap M = \seb x_M \sen$ for each $M \in \Mcal_f$.
Thus, $A = B'$, as desired.
\end{proof}

\begin{cor}\label{cor:minimaluniqueintersection}
Let $(X,f)$ be an invertible zero-dimensional system.
Then, there exists a minimal basic set $B$ such that
 $\abs{B \cap M} = 1$ for all $M \in \Mcal_f$.
Moreover, if $x_M \in B \cap M$ for each $M \in \Mcal_f$, then
 $B = \clos{\seb x_M \mid M \in \Mcal_f \sen}$.
\end{cor}
\begin{proof}
The combination of \cref{thm:extremal,thm:extremalunique}
 and \cref{cor:extremaluniqueminimal}
leads to the statements of this corollary.
\end{proof}

An extremal basic set need not be minimal, as the following example shows.

\begin{example}\label{example:non-minimal-extream}
We show that there exists a zero-dimensional system $(X,f)$
 such that if $B$ is an extremal basic set,
 $B$ is not minimal.
Let us arbitrarily fix an embedding of $X$ into $\R$ with
 respect to which $B$ is the extremal basic set, i.e.,
 $B = \inf_f$.
Suppose that $(X,f)$ contains two fixed points $p_1,p_2$
 and $\seb \seb p_1 \sen, \seb p_2 \sen \sen = \Mcal_f$.
In particular, it follows that $B_0 := \seb p_1, p_2\sen$ is
 the only minimal basic set.
Let us assume that there exist sequences
 $x_{i,n} \to p_i$ ($i = 1,2$)
 of points of $X$
 such that, for all $n = 1,2, \dotsc$, it follows that
 $\alpha(x_{1,n}) = \omega(x_{1,n}) = p_2$ and
 $\alpha(x_{2,n}) = \omega(x_{2,n}) = p_1$.
It is self-evident that such a system exists.
In this system, one of the $p_i$s is less than the other.
Without loss of generality, let us assume that $p_2 < p_1$.
For all sufficiently large $n$, we obtain that $x_{2,n} < p_1$.
Thus, $\inf_f(x_{2,n}) < p_1$ for such $n$s.
For the orbit $O(x_{2,n})$,
 there exists a sole accumulation point $p_1$.
It follows that $\inf_f(x_{2,n}) \in O(x_{2,n})$ for such $n$s.
Thus, we obtain $\inf_f(x_{2,n}) \in B \cap O(x_{2,n})$ and
 $\inf_f(x_{2,n}) \notin B_0$.
This shows that $B$ is not minimal, as desired.
\end{example}

Summarizing the argument above, we obtain the following. 

\begin{thm}\label{thm:main}
Let $(X,f)$ be an invertible zero-dimensional system.
Suppose that $(X,f)$ is densely aperiodic.
Then, there exists a minimal continuously decisive basic set $B$
 such that $\abs{M \cap B} = 1$ for each $M \in \Mcal_f$.
\end{thm}
\begin{proof}
By \cref{cor:minimaluniqueintersection}, there exists a minimal basic set
 $B$ such that $\abs{B \cap M} = 1$ for all $M \in \Mcal_f$.
By \cref{prop:minimalqscontdic},
 such a $B$ is continuously decisive because it is minimal as a quasi-section.
\end{proof}

Quasi-sections and basic sets might have been
 byproducts of the Bratteli--Vershik
 models.
However, for densely aperiodic zero-dimensional systems,
 a continuously decisive quasi-section or such a basic set
 determines the equivalence class of ordered Bratteli
 diagrams.
Although we could not present any applications in the present paper,
 the existence of basic sets that satisfy the condition of \cref{thm:main}
 must be related to some good properties of zero-dimensional systems,
 which merits further studies.

\vspace{5mm}

\noindent
\textsc{Acknowledgments:}
This work was partially supported by JSPS KAKENHI (Grant Number 20K03643).
%I would like to thank Editage (www.editage.jp)
% for providing English-language editing services.

\providecommand{\bysame}{\leavevmode\hbox to3em{\hrulefill}\thinspace}
\providecommand{\MR}{\relax\ifhmode\unskip\space\fi MR }
\providecommand{\MRhref}[2]{%
  \href{http://www.ams.org/mathscinet-getitem?mr=#1}{#2}
}
\providecommand{\href}[2]{#2}

%
%\bibliographystyle{amsalpha}
%%%\bibliography{test}
%\bibliography{My2021-02-06-bibtex}
\end{document}